\documentclass[12pt,reqno]{amsart}
\usepackage{amsmath, amsthm, amssymb}
\usepackage{hyperref}

\usepackage{multirow, array}
\usepackage{placeins}

\usepackage{mathtools}

\usepackage{caption} 
\advance \topmargin by -\headheight
\advance \topmargin by -\headsep
     
\evensidemargin \oddsidemargin
\newtheorem{thm}{Theorem}[section]
\newtheorem{lemma}[thm]{Lemma}

\newtheorem{cor}[thm]{Corollary}

\theoremstyle{definition}
\newtheorem{defn}[thm]{Definition}

\theoremstyle{remark}

\numberwithin{equation}{section}

\newcommand*\wrapletters[1]{\wr@pletters#1\@nil}
\def\wr@pletters#1#2\@nil{#1\allowbreak\if&#2&\else\wr@pletters#2\@nil\fi}

\def\alp{{\alpha}} 
\def\bet{{\beta}}  
\def\gam{{\gamma}} \def\Gam{{\Gamma}}
 \def\Del{{\Delta}}

\def\lam{{\lambda}}

\def\sig{{\sigma}}

\def\eps{\varepsilon}

\def\le{\leqslant} \def\ge{\geqslant}

\def\d{{\,{\rm d}}}

\def \sig{{\sigma}}

\def \bJ {\mathbb J}

\def \bN {\mathbb N}

\def \bQ {\mathbb Q}
\def \bR {\mathbb R}

\def \bZ {\mathbb Z}

\def \ba {\mathbf a}

\def \bu {\mathbf u}
\def \bv {\mathbf v}
\def \bx {\mathbf x}

\def \by {\mathbf y}

\def \bmu {{\boldsymbol{\mu}}}
\def \bbeta {\boldsymbol{\beta}}

\def \fm {\mathfrak m}
\def \fn {\mathfrak n}

\def \ft {\mathfrak t}

\def \fF {\mathfrak F}

\def \fJ {\mathfrak J}

\def \fM {\mathfrak M}
\def \fN {\mathfrak N}

\def \fU {\mathfrak U}

\def \cI {\mathcal I}
\def \cJ {\mathcal J}

\def \cL {\mathcal L}

\def \cR {\mathcal R}

\def \cU {\mathcal U}
\def \cV {\mathcal V}

\def \sinc {\mathrm{sinc}}
\def \meas {\mathrm{meas}}

\begin{document}
\title[Waring's problem with shifts]{Waring's problem with shifts}
\author[Sam Chow]{Sam Chow}
\address{School of Mathematics, University of Bristol, University Walk, Clifton, Bristol BS8 1TW, United Kingdom}
\email{Sam.Chow@bristol.ac.uk}
\thanks{}
\date{}
\begin{abstract} 
Let $\mu_1, \ldots, \mu_s$ be real numbers, with $\mu_1$ irrational. We investigate sums of shifted $k$th powers $\mathfrak{F}(x_1, \ldots, x_s) = (x_1 - \mu_1)^k + \ldots + (x_s - \mu_s)^k$. For $k \ge 4$, we bound the number of variables needed to ensure that if $\eta$ is real and $\tau > 0$ is sufficiently large then there exist integers $x_1 > \mu_1, \ldots, x_s > \mu_s$ such that $|\mathfrak{F}(\bx) - \tau| < \eta$. This is a real analogue to Waring's problem. When $s \ge 2k^2-2k+3$, we provide an asymptotic formula. We prove similar results for sums of general univariate degree $k$ polynomials.
\end{abstract}
\maketitle

\section{Introduction}
\label{intro}

Since its formulation \cite{War1770} in 1770, Waring's problem has been the benchmark for research on diophantine equations in many variables. Presently, inequalities of the shape
\[ |\lam_1 x_1^k + \ldots + \lam_s x_s^k| < \eta \]
enjoy a similar status in the world of diophantine inequalities. We proffer an alternate analogue to Waring's problem. Let $s$ and $k$ be positive integers, and let $\mu_1, \ldots, \mu_s$ be real numbers, with $\mu_1$ irrational. We investigate the values taken by sums of shifted $k$th powers
\[
\fF(x_1, \ldots, x_s) = (x_1 - \mu_1)^k + \ldots + (x_s - \mu_s)^k
\]
for integers $x_i > \mu_i$ ($1 \le i \le s$).

\begin{defn} \label{s1def}
For integers $k \ge 2$, let $s_1(k)$ be the least integer such that the following holds whenever $s \ge s_1(k)$. Let $\mu_1, \ldots, \mu_s$ be real numbers, with $\mu_1$ irrational. Let $\eta > 0$ be real number, and let $\tau$ be a sufficiently large positive real number. Then there exist integers $x_1 > \mu_1, \ldots, x_s>\mu_s$ such that
\begin{equation} \label{main}
|\fF(\bx) - \tau| < \eta.
\end{equation}
\end{defn}

\begin{thm} \label{s1bounds}  For $4 \le k \le 12$ we have $s_1(k) \le C_1(k)$, where $C_1(k)$ is given in the table below.
\FloatBarrier
\begin{table}[h!]
\caption{Upper bounds for $s_1(k)$.}
\label{table1}
\begin{center}
\begin{tabular*}{\textwidth}{@{} |r| @{\extracolsep{\fill} } ccccccccc|}
\hline
$k$    &$4$&$5$&$6$&$7$&$8$&$9$&$10$ & $11$ & $12$			 \\
\hline
$C_1(k)$ &$16$&$27$&$38$&$51$&$70$&$87$&$104$&$120$&$135$ \\
\hline
\end{tabular*}
\end{center}
\end{table}
\FloatBarrier
Further, if $k \ge 4$ then
\begin{equation} \label{largebound1}
s_1(k) < 4k \log k + (2 + 2 \log 2)k - 3. 
\end{equation}
\end{thm}

A simplification of our methods shows that $s_1(2) \le 5$, and the author showed in \cite{Cho2014b} that $s_1(3) \le 9$. Given more variables, we may obtain an asymptotic formula counting solutions to \eqref{main}. For positive real numbers $\tau$ and $\eta$, denote by $N(\tau) = N_{s,k}(\tau; \eta,\bmu)$ the number of integral solutions $\bx \in (\mu_1, \infty) \times \ldots \times (\mu_s, \infty)$ to \eqref{main}.

\begin{thm} \label{asymptotic}
Let $\eta > 0$. Let $k \ge 4$ and $s \ge 2k^2-2k+3$. Then 
\[ N(\tau) \sim 2 \eta \Gam(1+1/k)^s \Gam(s/k)^{-1} \tau^{s/k-1}. \]
\end{thm}

By a simplification of our methods, we may obtain a similar asymptotic formula for sums of five shifted squares, and the author showed in \cite{Cho2014b} that eleven variables suffice when $k=3$. Theorem \ref{asymptotic} implies that $s_1(k) \le 2k^2-2k+3$. We can achieve better bounds, even in a more general setting, at the cost of not having an asymptotic formula for $N(\tau)$. We introduce some definitions in order to state our results precisely. 

\begin{defn} \label{Indef}
Let $h_1,\ldots,h_s$ be degree $k$ polynomials with real coefficients. We say that
\[ H(\bx) = \sum_{i\le s} h_i(x_i) \]
is \emph{indefinite} if $k$ is odd, or if the leading coefficients of $h_1,\ldots,h_s$ do not all have the same sign.
\end{defn}

\begin{defn} \label{irr} Let $k \ge 2$. For $i=1,2,\ldots,s$, let $h_i(x)$ be a degree $k$ polynomial with real coefficients given by
\[ h_i(x) = \beta_{ik}x^k+ \ldots + \beta_{i1}x + \beta_{i0}. \]
The polynomials $h_1, \ldots, h_s$ satisfy the \emph{irrationality condition} if there exist $i_1, i_2 \in \{1,2,\ldots,s \}$ and $j_1,j_2 \in \{1,2,\ldots,k\}$ such that $\beta_{i_2j_2} \ne 0$ and $\beta_{i_1j_1} / \beta_{i_2j_2}$ is irrational.
\end{defn}

\begin{defn} \label{s0def}
For integers $k \ge 2$, let $s_0(k)$ be the least integer such that the following holds whenever $s \ge s_0(k)$. Let $h_1,\ldots,h_s \in \bR[y]$ be degree $k$ polynomials satisfying the irrationality condition, and put $H(\bx) = \sum_{i \le s} h_i(x_i)$. Let $\eta> 0$ and $\tau$ be real numbers, and assume that $H(\bx)$ is indefinite. Then there exists $\bx \in \bZ^s$ such that
\begin{equation} \label{main2}
|H(\bx) - \tau| < \eta. 
\end{equation}
\end{defn}

In \cite{Cho2014b}, the author showed that $s_0(3) \le 9$. Meanwhile, a result of Margulis and Mohammadi \cite[Theorem 1.4]{MM2011} implies that $s_0(2) \le 3$.  Freeman \cite[Theorem 1]{Fre2003} studied $s_0(k)$ as $k \to \infty$, demonstrating that $s_0(k)$ is dominated by a function that is asymptotic to $4k \log k$. Here we provide an exact bound. We can also achieve better upper bounds for $s_0(k)$ when a specific value of $k$ is given. 

\begin{thm} \label{s0bounds}
For $4 \le k \le 12$ we have $s_0(k) \le C_0(k)$, where $C_0(k)$ is given in the table below.
\FloatBarrier
\begin{table}[h!]
\caption{Upper bounds for $s_0(k)$.}
\label{table2}
\begin{center}
\begin{tabular*}{\textwidth}{@{} |r| @{\extracolsep{\fill} } ccccccccc|}
\hline
$k$   &$4$&$5$&$6$&$7$&$8$&$9$&$10$ & $11$ & $12$			 \\
\hline
$C_0(k)$ & $18$ & $29$ & $43$ & $59$ & $79$ & $99$ & $115$ & $132$ & $149$ \\
\hline
\end{tabular*}
\end{center}
\end{table}
\FloatBarrier
Further, if $k \ge 4$ then
\begin{equation} \label{largebound0}
s_0(k) < 4k \log k + (2 + 2 \log 2)k - 3. 
\end{equation}
\end{thm}

Our overall strategy is to use the Davenport-Heilbronn method, in the style of Freeman \cite{Fre2003}. A classical major and minor arc dissection is also needed, the point being that either a Weyl sum is small or its coefficients have good simultaneous rational approximations. To deduce Theorems \ref{s1bounds} and \ref{s0bounds}, we restrict some of the variables to lie in `diminishing ranges', an idea that goes back to Hardy and Littlewood \cite{HL1925}. For the remaining variables, we use known analogues to Weyl's inequality (see Lemma \ref{ClassicalMajorIngredient}). These save a power of $\sigma(k)$ per variable on classical minor arcs where, for $d \in \bN$, we define
\begin{equation} \label{sigmaDef} \sigma(d)^{-1} = \begin{dcases}
2^{d-1}, & d \le 8 \\
4(d^2 - 3d + 3), & d \ge 9.
\end{dcases}
\end{equation}
Note that $\sigma(d)$ is a decreasing function. We will ultimately deduce the following bounds, which are responsible for many of our results.

\begin{thm} \label{basicThm}
Let $k \ge 4$ and $t$ be positive integers, and put
\[ E = 1 +  \max(2k-2, \lfloor k(1-1/k)^t \sigma(k)^{-1}  \rfloor ) . \]
Then
\begin{equation} \label{basicBounds}
s_\iota(k) \le 2t + E \qquad (\iota = 0,1).
\end{equation}
\end{thm}

Choosing $t$ optimally gives Table \ref{table2}. Moreover, specialising 
\[
t=\lceil 2k \log k + k \log 2 \rceil
\]
in Theorem \ref{basicThm} and using the fact that 
\[ (1-1/k)^k \le 1/e\]
yields \eqref{largebound1} and \eqref{largebound0} when $k \ge 9$. For $4 \le k \le 8$, these inequalities follow from Tables \ref{table1} and \ref{table2}.

The methods developed in \cite{Cho2014b}, based on low moment estimates for Weyl sums, allow us to obtain better bounds for $s_1(k)$, particularly if $k$ is not too large. Owing to the technical nature of our general bound, we defer this until \S \ref{mainProof}. From the above discussion, it remains for us to establish Theorem \ref{asymptotic}, Theorem \ref{basicThm}, and the bounds implicit in Table \ref{table1}.

The bounds $C_1(k)$ given in Theorem \ref{s1bounds}, for the most part, fall far short of the corresponding records for Waring's problem, which we list below (see \cite{VW2002}). Here $G(k)$ is the least positive integer $s$ such that every sufficiently large positive integer is a sum of at most $s$ $k$th powers of positive integers, and $B(k)$ is the best known upper bound for $G(k)$.
\FloatBarrier
\begin{table}[h!]
\caption{Upper bounds for $s_1(k)$ versus upper bounds for $G(k)$.}
\label{table3}
\begin{center}
\begin{tabular*}{\textwidth}{@{} |r| @{\extracolsep{\fill} } ccccccccc|}
\hline
$k$   &$ 4$&$5$&$6$&$7$&$8$&$9$&$10$ & $11$ & $12$			 \\
\hline
$C_1(k)$ & 16&27&38&51&70&87&104&120&135 \\
\hline
$B(k)$ & 16 & 17 & 24 & 33 & 42 & 50 & 59 & 67 & 76 \\
\hline
\end{tabular*}
\end{center}
\end{table}
\FloatBarrier
The discrepancy is not surprising, since divisibility cannot be used in the inequalities case. One could argue that a fairer comparison would be to \cite{Vau1986jlms} and \cite{Vau1986ProcLMS}, which predated smooth numbers. There the bounds 
\[ G(5) \le 21, \quad G(6) \le 31, \quad G(7) \le 45, \quad G(8) \le 62, \quad G(9) \le 82 \]
are given, and Theorem \ref{s1bounds} provides comparable bounds.

Theorem \ref{asymptotic} uses Wooley's work \cite{Woo2014} on Vinogradov's mean value theorem, via the proof of \cite[Lemma 5.3]{BKW2010}. The theorem is then established via a recipe developed by Freeman \cite{Fre2002} and Wooley \cite{Woo2003}. Since an asymptotic formula is sought, diminishing ranges cannot be used here. The number of variables needed in Theorem \ref{asymptotic}, namely $2k^2-2k+3$, slightly exceeds the number currently required to obtain an asymptotic formula in Waring's problem (see \cite[Theorems 1.3 and 1.4]{Woo2014}). This is to be expected, since the latter uses Hua's lemma as an input, and we cannot use this since our polynomials may be irrational.

The work underpinning Theorem \ref{basicThm} will come as no surprise to Freeman's devotees. Indeed, we follow \cite{Fre2003}, with the only additional ingredient being Wooley's latest efficient congruencing work \cite{Woo2014}. The main purpose of this paper is to seek the best upper bounds on $s_1(k)$. We find that we can achieve better bounds on $s_1(k)$ by using more slowly diminishing ranges. These ranges were used in \cite{Vau1986ProcLMS}, however the methods that we use to produce low moment estimates for these ranges are necessarily different.

As the history of Waring's problem is well known, we merely point the reader towards \cite{VW2002}. For inhomogeneous additive diophantine equations, one can consult \cite[\S 11.4 and \S 12.4]{Nat2000}. Most research on diophantine inequalities has focussed on additive forms (see \cite{BKW2010} for a summary). At the opposite extreme, real forms have been considered in the most general settings (see \cite{Fre2000} and \cite{Sch1980}). As discussed, Freeman \cite{Fre2003} studied additive inhomogeneous polynomials. Other specialisations include Harvey's work \cite{Har2011} involving norm forms and the author's work \cite{Cho2014a} on split forms.

Since Margulis' resolution of the Oppenheim conjecture \cite{Mar1989}, dynamical techniques have proven to be a highly effective means of tackling quadratic diophantine inequalities. As mentioned, Margulis and Mohammadi \cite[Theorem 1.4]{MM2011} have generalised Margulis' result, showing that $s_0(2) \le 3$. G\"otze \cite{Goe2004} handled positive definite quadratic forms in five or more variables, thereby proving the Davenport-Lewis conjecture. As far as the author is aware, sums of shifted powers were first considered by Marklof \cite{Mar2003}, who studied sums of shifted squares in relation to the Berry-Tabor conjecture from quantum chaos. The cubic case was discussed in \cite{Cho2014b}, and here we examine quartic and higher degree polynomials.

We now expound on our strategy for proving Theorem \ref{basicThm} and the bounds implicit in Table \ref{table1}. The treatment of the Davenport-Heilbronn major arc is relatively standard. Classical minor arcs are treated using low moment estimates involving diminishing ranges. Our simultaneous rational approximations allow us to develop an $\eps$-free analogue to Hua's lemma on classical major arcs (see Lemma \ref{ClassicalMajorIngredient} and Corollary \ref{ClassicalMajorGeneral}). Finally, we invoke \cite[Lemmas 8 and 9]{Fre2003} to obtain a nontrivial upper bound on Davenport-Heilbronn minor arcs. The low moment estimates required for Theorem \ref{basicThm} come almost for free, due to the nature of the diminishing ranges. We use more ambitious diminishing ranges to obtain the bounds implicit in Table \ref{table1}. There we classify our classical major arcs according to the size of the denominator, and then apply different techniques appropriately.

The plan for this paper is as follows. In \S \ref{prelim}, we present work of Freeman which exploits the irrationality of $\mu_1$, introduce our main kernel function, analyse classical major arcs in some generality, and apply Wooley's work on Vinogradov's mean value theorem. In \S\S \ref{AF}--\ref{mainProof} we prove Theorem \ref{asymptotic}, Theorem \ref{basicThm}, and the bounds implicit in Table \ref{table1}, respectively.

We adopt the convention that $\eps$ denotes an arbitrarily small positive number, so its value may differ between instances. Bold face will be used for vectors, for instance we shall abbreviate $(x_1,\ldots,x_s)$ to $\bx$. For real numbers $x$, we denote by $\lfloor x \rfloor$ the greatest integer $n$ such that $n \le x$. We shall use the unnormalised sinc function, given by $\sinc(x) = \sin(x)/x$ for $x \in \bR \setminus \{0\}$ and $\sinc(0) = 1$. The pronumeral $P$ will always denote a large positive real number. We shall use $g(\alp)$, $g_i(\alp)$ and $f_j(\alp)$ to denote Weyl sums, to be explicitly defined in each situation. 

The author thanks Trevor Wooley for suggesting this line of research, as well as for his dedicated supervision. Thanks also to an anonymous referee for carefully reading this manuscript.

\section{Preliminaries}
\label{prelim}

The following observation is the starting point for some of our inductive proofs.

\begin{lemma} \label{SecondMoment}
Let $h$ be a real polynomial of degree $d \ge 2$. Let $x$ and $y$ be integers such that $x,y > P$ and
\[
|h(x) - h(y)| < \eta.
\]
Then $x=y$.
\end{lemma}

\begin{proof} The mean value theorem gives
\[ (h(x)-h(y)) / (x-y) \gg P^{d-1}, \]
so $|x-y| < 1$.
\end{proof}

We will require Freeman's bounds on Davenport-Heilbronn minor arcs. In \cite[Lemmas 8 and 9]{Fre2003}, the underlying variables lie in the range $(0,P]$. The same results hold, with the same proof, when the underlying variables lie in $(bP,cP]$ for some fixed real numbers $b \ge 0$ and $c > b$. We summarise this in the lemma below. For $h \in \bR[x]$, and for real numbers $b \ge 0$ and $c > b$, we shall write
\[
g_{b,c}(\alp; h) = \sum_{bP < x \le cP} e(\alp h(x)).
\]
Note the identity
\[
g_{b,c}(\alp; h) = g_{0,c}(\alp; h) - g_{0,b}(\alp; h),
\]
which will be used to infer certain bounds.

\begin{lemma} \label{Freeman}
Let $k \ge 2$ be an integer, let $\xi < 1$ be a positive real number, and let $h_1, h_2 \in \bR[y]$ be degree $k$ polynomials satisfying the irrationality condition. Let $0 \le b < c$, and let
\[ 
g_i(\alp) = g_{b,c}(\alp; h_i), \qquad i=1,2. 
\]
Then there exists a positive real-valued function $T(P)$ such that 
\[ \lim_{P \to \infty} T(P) = \infty\]
and
\begin{equation} \label{FreemanEq}
\sup_{ P^{\xi-k} \le |\alp| \le T(P)} 
|g_1(\alp) g_2(\alp)|
\ll P^2T(P)^{-1}.
\end{equation}
\end{lemma}

This may appear stronger than Freeman's conclusion that
\begin{equation} \label{Freeman0} 
\sup_{ P^{\xi-k} \le |\alp| \le T(P)} |g_1(\alp) g_2(\alp)| = o(P^2).
\end{equation}
However, the bound \eqref{Freeman0} gives a positive real-valued function $T_1(P)$ such that 
\[ \lim_{P \to \infty} T_1(P) = \infty \]
and
\[
\sup_{ P^{\xi-k} \le |\alp| \le T(P)} |g_1(\alp) g_2(\alp)| \ll P^2 T_1(P)^{-1}.
\]
By putting $T_0(P) = \min(T(P),T_1(P))$, we obtain \eqref{FreemanEq} with $T_0(P)$ in place of $T(P)$. 

We will make particular use of the kernel function
\begin{equation*}
K(\alpha) = K_\eta(\alp)= \eta \cdot \sinc^2(\pi \alp \eta).
\end{equation*}
This was first used by Davenport and Heilbronn \cite{DH1946}. It satisfies
\begin{equation} \label{Kbound} 
0 \le K(\alp) \ll \min(1, |\alp|^{-2})
\end{equation} 
and, for any real number $t$,
\begin{equation}\label{orth1}
\int_\bR e(\alp t) K(\alp) \d \alp = \max(0, 1- |t/\eta|).
\end{equation}
Similarly
\begin{equation}\label{orth2}
4\int_\bR e(\alp t) K(2\alp) \d \alp = \max(0, 2- |t/\eta|).
\end{equation}
For $\kappa > 0$, we define the indicator function
\begin{equation} \label{ind}
U_\kappa(t)  = \begin{cases}
1, &\text{if } |t| < \kappa\\
0, &\text{if } |t| \ge \kappa.
\end{cases}
\end{equation}
By \eqref{orth1} and \eqref{orth2} we have
\begin{equation} \label{OrthBounds}
0 \le \int_\bR e(\alp t) K(\alp) \d \alp \le U_\eta(t) \le 4\int_\bR e(\alp t) K(2\alp) \d \alp \le 2 U_{2\eta}(t).
\end{equation}

The following lemma is integral to our classical major arc analysis. The idea is that if a Weyl sum is large then its coefficients have good simultaneous rational approximations. Given such rational approximations, we can follow a standard procedure to bound the Weyl sum. Recall \eqref{sigmaDef}.

\begin{lemma} \label{ClassicalMajorIngredient}
Let $d \ge 2$ be an integer, and let $h \in \bR[x]$ be a monic polynomial of degree $d$. Let $0 \le b < c$, and let $g(\alp) = g_{b,c}(\alp; h)$. Let $\alp \in \bR$, and assume that
\[ |g(\alp)| > P^{1-\sigma(d)+\eps}. \]
Then there exist relatively prime integers $a$ and $q$ such that
\begin{equation} \label{bbound} 
0 < q < P^{d\sigma(d)}, \qquad |q \alp -a| < P^{d \sigma(d)-d}
\end{equation}
and
\[
g(\alp) \ll q^{\eps - 1/d}P (1+P^d|\bet|)^{-1/d},
\]
where $\bet = \alp - a/q$. The implicit constant does not depend on the coefficients of $h$.
\end{lemma}

\begin{proof} Define $\alp_0, \ldots, \alp_d$ by
\[ 
\alp h(x) = \sum_{i=0}^d \alp_i x^i, 
\]
and note that $\alp_d = \alp$. For $q \in \bN$ and $\bv \in \bZ^d$, put
\[ S(q,\bv) = \sum_{x=1}^q e((v_d x^d + \ldots+v_1x)/q). \]
For $\beta_0, \beta_1, \ldots, \beta_d \in \bR$, put
\[
I(\bbeta) = \int_{bP}^{cP} e(\bet_d x^d + \ldots + \bet_1 x + \bet_0) \d x.
\]

First assume that $d \le 8$. At least one of $|g_{0,b}(\alp; h)|$ and $|g_{0,c}(\alp; h)|$ must exceed $\frac12 P^{1-\sigma(d)+\eps}$. Thus, by \cite[Theorem 5.1]{Bak1986}, there exist $q_0 \in \bN$ and $\bv \in \bZ^d$ such that
\[
q_0 < P^{d \sigma(d)}, \qquad (q_0, v_d, \ldots, v_1) = 1, \qquad (q_0,v_d, \ldots, v_2) < P^\eps
\]
and
\[
|q_0 \alp_j - v_j| < P^{d\sigma(d)- j} \qquad (1 \le j \le d).
\]
Put $\bet_0 = \alp_0$ and
\[ \bet_j = \alp_j - q_0^{-1} v_j \qquad (1 \le j \le d). \]
It follows from \cite[Lemma 4.4]{Bak1986} that
\[
g(\alp) - q_0^{-1}S(q_0,\bv)I(\bbeta) \ll q_0^{1-1/d} P^\eps.
\]
Now \cite[Theorems 7.1 and 7.3]{Vau1997} give
\begin{align*} g(\alp) &\ll  q_0^{1-1/d}P^\eps + q_0^{\eps-1/d}P(1+P^d|\bet_d|)^{-1/d} 
\\ 
&\ll q_0^{\eps-1/d}P(1+P^d|\bet_d|)^{-1/d} . 
\end{align*}
The integers $a = v_d / (q_0, v_d)$ and $q = q_0 / (q_0, v_d)$ have the desired properties.

It remains to consider $d \ge 9$. From \cite[Theorem 11.2]{Woo2014} we deduce that there exist $r \in \bN$ and $\ba \in \bZ^d$ such that $r < P^{d \sigma(d)}$
and
\[
|r \alp_j - a_j| < P^{d\sigma(d)- j} \qquad (1 \le j \le d).
\]
Now \cite[Lemma 4.6]{Bak1986} gives $q_0 \in \bN$ and $\bv \in \bZ^d$ such that
\[ 
q_0 < P^{d \sigma(d)}, \qquad (q_0, v_d, \ldots, v_2) \le 2d^2
\]
and
\[ 
|q_0 \alp_j - v_j| < P^{d\sigma(d)- j} \qquad (1 \le j \le d).
\]
Dividing by $(q_0, v_d, \ldots, v_1)$, we may assume further that $(q_0, v_d, \ldots, v_1) = 1$. The rest of the proof is identical to the case $d \le 8$.
\end{proof}

This allows us to formulate an $\eps$-free analogue to Hua's lemma on classical major arcs. 

\begin{cor} \label{ClassicalMajorGeneral} Let $k \ge 2$ be an integer, and let $u > 2k$ be a real number. Fix a degree $k$ polynomial $h \in \bR[x]$, and fix $L>0$.  Let $0 \le b < c$, and let $g(\alp) = g_{b,c}(\alp; h)$. Put
\[ \fN = \{ \alp \in \bR : |g(\alp)| > P^{1-\sigma(k)+\eps} \}, \]
and let $\fU$ be the intersection of $\fN$ with an interval of length $L$. Then
\begin{equation} \label{cmgeq} 
\int_\fU |g(\alp)|^u \d \alp \ll_{h,L} P^{u-k}.
\end{equation}
\end{cor}

\begin{proof} By changing variables, we may assume that $h$ is monic. We then apply Lemma \ref{ClassicalMajorIngredient} with $d=k$; note that if $q \in \bN$ then there are $O_L(q)$ integers $a$ satisfying \eqref{bbound} for some $\alp \in \fU$. Hence
\[
\int_\fU |g(\alp)|^u \d \alp \ll 
P^u \sum_{q < P^{k\sigma(k)}} q^{1+\eps-u/k} J,
\]
where
\[  
J = \int_0^\infty (1+P^k \bet)^{-u/k} \d \bet \ll P^{-k}.
\]
As $u > 2k$ and $\eps$ is small, we now have \eqref{cmgeq}.
\end{proof}

Lastly, we will need the following application of Wooley's work on Vinogradov's mean value theorem.

\begin{lemma} \label{EClem}
Let $\phi$ be a polynomial of degree $k \ge 3$ with real coefficients, let $X$ be a large positive real number, and let $\eta$ be a positive real number. Let $t \ge k^2-k+1$ be an integer, and let $U_{\phi,t}(X)$ denote the number of integer solutions to the inequality
\[ \Bigl| \sum_{j\le t} (\phi(x_j) - \phi(y_j)) \Bigr| < \eta \]
with $1 \le x_j, y_j \le X$. Then
\[ U_{\phi,t}(X) \ll X^{2t-k+\eps}. \]
\end{lemma}

\begin{proof} The proof of \cite[Lemma 5.3]{BKW2010} shows that
\[ 
U_{\phi,t}(X) \ll X^{k(k-1)/2}J_{t,k}(X), 
\]
where $J_{t,k}(X)$ is the number of integer solutions to the system
\[ 
\sum_{j \le t} (x_j^l - y_j^l) = 0 \qquad (1 \le l \le k) 
\]
with $1 \le x_j, y_j \le X$. Moreover, from \cite[Corollary 1.2]{Woo2014} we have
\[ J_{t,k}(X) \ll X^{2t-k(k+1)/2+\eps}, \]
completing the proof.
\end{proof}

\section{An asymptotic formula}
\label{AF}

In this section we prove Theorem \ref{asymptotic}. Let $\tau$ be a large positive real number, and put
\begin{equation} \label{Pdef}
P=\tau^{1/k}.
\end{equation}
We may plainly assume that
\begin{equation} \label{wlog}
0 \le \mu_1, \ldots, \mu_s < 1.
\end{equation}
One can easily check that
\[ N(\tau) - N^*(\tau) \ll P^{(k-1)(s-1)/k} = o(\tau^{s/k-1}),\]
where $N^*(\tau)$ is the number of integral $\bx \in [1,P]^s$ satisfying \eqref{main}. It therefore suffices to prove the theorem  with $N^*(\tau)$ in place of $N(\tau)$.

For $i=1,2,\ldots,s$, put
\[ g_i(\alp) = \sum_{x \le P} e(\alp(x-\mu_i)^k). \]
Let $0 < \xi < 1$, and recall that $\mu_1 \notin \bQ$. With $T(P)$ as in Lemma \ref{Freeman}, applied to the polynomials $(x-\mu_1)^k$ and $(x-\mu_2)^k$, we define our Davenport-Heilbronn major arc by
\begin{equation} \label{dh1}
\fM = \{ \alp \in \bR : |\alp| \le P^{\xi - k} \}, 
\end{equation}
our minor arcs by
\begin{equation} \label{dh2}
\fm = \{ \alp \in \bR : P^{\xi - k} < |\alp| \le T(P) \}, 
\end{equation}
and our trivial arcs by
\begin{equation} \label{dh3}
\ft = \{ \alp \in \bR :  |\alp| > T(P) \}.
\end{equation}

Next we deploy a kernel function introduced in \cite[\S 2.1]{Fre2002}. Put 
\begin{equation} \label{Ldef}
L(P) = \min(\log T(P), \log P), \qquad \delta = \eta L(P)^{-1}
\end{equation}
and
\[ K_{\pm}(\alp) = \frac {\sin(\pi \alp \delta) \sin(\pi \alp(2 \eta \pm \delta))} {\pi^2 \alp^2 \delta}. \]
From \cite[Lemma 1]{Fre2002} and its proof, we have
\begin{equation} \label{Kbounds}
K_\pm(\alp) \ll \min(1, L(P) |\alp|^{-2})
\end{equation}
and
\begin{equation} \label{Ubounds}
0 \le \int_\bR e(\alp t) K_{-}(\alp)\d\alp \le U_\eta(t) \le \int_\bR e(\alp t) K_{+}(\alp)\d\alp \le 1,
\end{equation}
where we recall the definition \eqref{ind}. Moreover, the expression
\[ \Bigl|  \int_\bR e(\alp t) K_\pm(\alp) \d \alp - U_\eta(t) \Bigr| \]
is less than or equal to 1, and is equal to 0 whenever $| | t|- \eta| > \eta L(P)^{-1}$. 

It will be convenient to work with nonnegative kernels in part of the analysis, as in \cite[\S 2]{PW2013}. We note that
\begin{equation} \label{decompose}
|K_{\pm}(\alp)|^2 = K_1(\alp)K_2^{\pm}(\alp),
\end{equation}
where
\[ K_1(\alp) = \sinc^2(\pi \alp \delta) \]
and
\[ K_2^{\pm}(\alp) = (2\eta \pm \delta)^2 \sinc^2(\pi \alp (2\eta \pm \delta)). \]
As \eqref{OrthBounds} holds for all $\eta > 0$, we also have
\begin{equation} \label{aux1}
0 \le \int_\bR e(\alp t) K_1(\alp) \d \alp \le \delta^{-1} U_\delta(t) \ll L(P) \cdot U_\delta(t)
\end{equation}
and
\begin{equation} \label{aux2}
0 \le \int_\bR e(\alp t) K_2^{\pm}(\alp) \d \alp \le (2\eta \pm \delta) U_{2\eta \pm \delta}(t) \ll U_{2 \eta \pm \delta}(t).
\end{equation}

From \eqref{Ubounds} we have
\[
R_{-}(P)  \le N^*(\tau) \le R_+(P), 
\]
where 
\[ R_\pm(P) = \int_\bR g_1(\alp) \cdots g_s(\alp) e(-\alp \tau) K_\pm(\alp) \d \alp. \]
It therefore remains to show that
\begin{equation} \label{goal3HP} 
R_\pm(P) = 2 \eta \Gam(1+1/k)^s \Gam(s/k)^{-1} P^{s-k} + o(P^{s-k}).
\end{equation}
We begin by demonstrating the bound
\begin{equation} \label{rest3HP}
\int_{\fm \cup \ft} g_1(\alp) \cdots g_s(\alp) e(-\alp \tau)K_\pm(\alp) \d \alp = o(P^{s-k}).
\end{equation}
For this purpose it suffices, by symmetry and H\"older's inequality, to prove that
\begin{equation} \label{rest3altHP}
\int_{\fm \cup \ft} |g_1(\alp)g_2(\alp)g_3(\alp)^{s-2} K_\pm(\alp)| \d \alp = o(P^{s-k}).
\end{equation}

Recalling \eqref{sigmaDef}, let
\[ \fN = \{ \alp \in \bR : |g_3(\alp)| > P^{1- \sig(k)+\eps} \}, \]
put $\fn = \bR \setminus \fN$, and let $\fU$ be the intersection of $\fN$ with a unit interval. For subsets $U \subseteq \bR$, write 
\[ \cI_\pm(U) = \int_U |g_1(\alp)g_2(\alp)g_3(\alp)^{s-2} K_\pm(\alp)| \d \alp.\]

By assumption we have $s \ge 2k^2 - 2k + 3$. Thus, by \eqref{aux1}, \eqref{aux2}, Lemma \ref{EClem} and a trivial estimate we have
\[ 
\int_\bR |g_i(\alp)|^{s-1} K_1(\alp) \d \alp \ll P^{s-1-k+\eps}L(P) 
\]
and
\[ \int_\bR |g_i(\alp)|^{s-1} K_2^\pm (\alp) \d \alp \ll P^{s-1-k+ \eps} \]
for $i = 1,2,3$. Cauchy's inequality, \eqref{Ldef} and \eqref{decompose} now give
\[ 
\int_\bR |g_i(\alp)^{s-1} K_\pm(\alp)| \d \alp \ll P^{s-1-k+ 2\eps} \qquad (i=1,2,3).
\]
Therefore, by H\"older's inequality, we have
\[ \int_\bR |g_1(\alp) g_2(\alp)  g_3(\alp)^{s-3} K_\pm(\alp)| \d \alp \ll P^{s-1-k+\eps}, \]
so
\begin{align} \notag
\cI_\pm(\fn) \ll (\sup_{\alp \in \fn} |g_3(\alp)|) \cdot P^{s-1-k+\eps}
\\ 
\label{ClassicalMinor3HP} \le P^{s-k-\sigma(k)+2\eps} = o(P^{s-k}).
\end{align}

Combining Corollary \ref{ClassicalMajorGeneral} with \eqref{FreemanEq} gives
\begin{align*}
\int_{\fm \cap \fU}  |g_1(\alp)g_2(\alp)g_3(\alp)^{s-2}| \d \alp 
&\ll (\sup_{\alp \in \fm} |g_1(\alp) g_2(\alp)|) \cdot P^{s-2-k} \\
&\ll P^{s-k} T(P)^{-1}
\end{align*}
which, recalling \eqref{Ldef} and \eqref{Kbounds}, yields
\begin{equation} \label{MinorMajor3HP}
\cI_\pm(\fm \cap \fN) \ll P^{s-k}T(P)^{-1}L(P) = o(P^{s-k}).
\end{equation}
By Corollary \ref{ClassicalMajorGeneral} and a trivial estimate, we have
\[
\int_\fU |g_1(\alp)g_2(\alp)g_3(\alp)^{s-2}| \d \alp \ll P^{s-k}.
\]
In light of \eqref{Ldef} and \eqref{Kbounds}, we now have
\begin{equation} \label{TrivialMajor3HP}
\cI_\pm(\ft \cap \fN) \ll P^{s-k} \sum_{n=0}^\infty L(P)\cdot (T(P)+n)^{-2}  = o(P^{s-k}).
\end{equation}
Since 
\[
\fm \cup \ft \subseteq \fn \cup (\fm \cap \fN) \cup (\ft \cap \fN),
\]
the inequalities \eqref{ClassicalMinor3HP}, \eqref{MinorMajor3HP} and \eqref{TrivialMajor3HP} give \eqref{rest3altHP}, which in particular establishes \eqref{rest3HP}.

Next we consider
\begin{equation} \label{I1def}
\cI_\pm^{(1)} = \int_\fM g_1(\alp) \cdots g_s(\alp) e(-\alp \tau) K_\pm (\alp) \d \alp,
\end{equation}
following the recipe given in \cite[\S 6]{Woo2003}. Define
\[ I(\alp) = \int_0^P e(\alp x^k) \d x,\]
\[ \cI_\pm^{(2)} = \int_\fM I(\alp)^s e(-\alp \tau) K_\pm (\alp) \d \alp \]
and 
\[ \cI_\pm^{(3)} = \int_\bR I(\alp)^s e(-\alp \tau) K_\pm (\alp) \d \alp. \]
By \cite[Lemma 4.4]{Bak1986}, if $\alp \in \fM$ and $1 \le i \le s$ then
\[
g_i(\alp) = \int_0^P e(\alp(x-\mu_i)^k)\d x+ O(1) = I(\alp)+O(1).
\]
Recalling \eqref{Kbounds}, we now conclude that
\begin{equation} \label{d12}
\cI^{(1)}_\pm - \cI^{(2)}_\pm \ll 
P^{\xi-k}P^{s-1} = o(P^{s-k}),
\end{equation}
since $\xi < 1$. Moreover, it follows from \cite[Theorem 7.3]{Vau1997} that
\[ 
I(\alp) \ll |\alp|^{-1/k}, 
\]
so by \eqref{Kbounds} we have
\begin{equation} \label{d23} \cI^{(3)}_\pm - \cI^{(2)}_\pm 
\ll \int_{P^{\xi-k}}^\infty \alp^{-s/k} \d \alp = o(P^{s-k}).
\end{equation}

The final step is to provide asymptotics for
\[ 
\cI_\pm^{(3)} = \int_{(0,P]^s} \int_\bR  e(\alp(x_1^k + \ldots + x_s^k - \tau)) K_\pm(\alp) \d \alp \d \bx . 
\]
Changing variables with $u_i = P^{-k} x_i^k$ ($1 \le i \le s$) yields
\[ \cI_\pm^{(3)} = k^{-s}P^s \int_{(0,1]^s} (u_1 \cdots u_s)^{1/k-1} \Delta_\pm(\bu) \d \bu, \]
where 
\[ \Delta_\pm(\bu) =
\int_\bR e(\alp(P^k(u_1+\ldots+u_s)-\tau))K_\pm(\alp) \d \alp.
\]
Put
\[
\Delta^*(\bu) = \begin{cases}
1, &\text{if } |u_1+\ldots+u_s -1| < \eta P^{-k} \\
0,& \text{if } |u_1+\ldots+u_s -1| \ge \eta P^{-k}
\end{cases}
\]
and
\[ I^* = \int_{(0,1]^s} (u_1 \cdots u_s)^{1/k-1} \Delta^*(\bu) \d \bu. \]

In light of \eqref{Pdef} and the discussion following \eqref{Ubounds}, we see that
\[ \Delta_\pm(\bu) = \Delta^*(\bu), \]
except possibly when
\begin{equation} \label{except}
||u_1+\ldots+u_s - 1| - \eta P^{-k}| \le \eta P^{-k} L(P)^{-1},
\end{equation}
in which case we have $|\Delta_\pm(\bu) - \Delta^*(\bu)| \le 1$. If \eqref{except} is satisfied then there exists $j \in \{1,2,\ldots, s\}$ such that $u_j \gg 1$. For $j=1,2,\ldots,s$, let $T_j$ denote the set of $\bu \in (0,1]^s$ satisfying \eqref{except} and $u_j \gg 1$. Now
\[ \int_{T_j} (u_1 \cdots u_s)^{1/k-1} \d \bu \ll P^{-k} L(P)^{-1} \qquad (1 \le j \le s), \]
so
\[ \int_{(0,1]^s} (u_1 \cdots u_s)^{1/k-1}(\Delta_\pm(\bu) - \Delta^*(\bu)) \d \bu = o(P^{-k}). \]
Thus,
\begin{equation} \label{IndDif}
\cI_\pm^{(3)}-k^{-s}P^s  I^* = o(P^{s-k}).
\end{equation}

For $\bu \in (0,1]^s$, write $\bu' = (u_1, \ldots, u_{s-1})$ and $Y = 1- u_1 - \ldots - u_{s-1}$. For $S \subseteq (0,1]^s$, define
\[ I(S) = \int_S (u_1 \cdots u_s)^{1/k-1} \Delta^*(\bu) \d \bu. \]
Let $I_1 = I((0,1]^{s-1} \times (0,P^{-1}))$ and $I_2 = I((0,1]^{s-1} \times [P^{-1},1])$,
so that 
\begin{equation} \label{tweak}
I^* = I_1 + I_2.
\end{equation}
First we show that 
\begin{equation} \label{I1}
I_1 = o(P^{-k}).
\end{equation}
Since $\int_0^{P^{-1}} u_s^{1/k-1} \d u_s = o(1)$, it suffices for \eqref{I1} to show that
\begin{equation} \label{I1g}
\int_{(0,1]^{s-1}} (u_1 \cdots u_{s-1})^{1/k-1} \Delta^*(\bu) \d \bu' \ll P^{-k},
\end{equation}
uniformly for $u_s \in (0,P^{-1})$. Let $0 < u_s < P^{-1}$. If $\Delta^*(\bu) = 1$ then there exists $j \in \{1,2,\ldots, s-1\}$ such that $u_j \gg 1$. For $j=1,2,\ldots,s-1$, let $R_j$ denote the set of $\bu' \in (0,1]^{s-1}$ such that $\Delta^*(\bu) = 1$ and $u_j \gg 1$. Now
\[ \int_{R_j} (u_1 \cdots u_{s-1})^{1/k-1} \d \bu' \ll P^{-k} \qquad (1 \le j \le s-1), \]
which establishes \eqref{I1g} and in particular \eqref{I1}.

If $\Delta^*(\bu) = 1$ and $u_s \ge P^{-1}$ then $|u_s - Y| < \eta  P^{-k}$ so, by the mean value theorem,
\[ u_s^{1/k-1} - Y^{1/k-1} \ll (P^{-1})^{1/k-2}P^{-k} = P^{2-k-1/k} = o(1). \]
Combining this with the bound
\[ \int_{(0,1]^{s-1}} (u_1 \cdots u_{s-1})^{1/k-1} \int_{P^{-1}}^1 \Delta^*(\bu) \d u_s \d \bu' \ll P^{-k} \]
gives
\begin{equation} \label{I23}
I_2 - I_3 = o(P^{-k}),
\end{equation}
where 
\[ I_3 = \int_{(0,1]^{s-1}} (u_1 \cdots u_{s-1}Y)^{1/k-1}  \int_{P^{-1}}^1 \Delta^*(\bu) \d u_s \d \bu'. \]
Let $R$ be the set of $\bu' \in (0,1]^{s-1}$ such that $Y>0$. As $|u_s - Y| < \eta P^{-k}$ whenever $\Delta^*(\bu) \ne 0$, we have
\[
I_3 = \int_R (u_1 \cdots u_{s-1}Y)^{1/k-1}  \int_{P^{-1}}^1 \Delta^*(\bu) \d u_s \d \bu'.
\]

Next we show that
\begin{equation} \label{I34}
I_3 - I_4 = o(P^{-k}),
\end{equation}
where
\begin{align*} \notag I_4 &= \int_R (u_1 \cdots u_{s-1}Y)^{1/k-1}  \int_\bR\Delta^*(\bu) \d u_s \d \bu'  \\
&= 2 \eta P^{-k} \int_R (u_1 \cdots u_{s-1}Y)^{1/k-1} \d \bu'.
\end{align*}
Let $\bu \in R\times \bR$ be such that $\Delta^*(\bu) = 1$. Then $|u_s - Y| < \eta P^{-k}$, so $u_s > -\eta P^{-k}$. If $u_s < P^{-1}$ then $Y < 2P^{-1}$ and $u_j \gg 1$ for some $j \in \{1,2,\ldots,s-1\}$, so we can change variables from $u_j$ to $Y$ to show that the contribution from these $\bu$ is $o(P^{-k})$. Meanwhile, if $u_s > 1$ then $Y > 1 - \eta P^{-k}$ and $u_1, \ldots, u_{s-1} \ll P^{-k}$, so the contribution from these $\bu$ is also $o(P^{-k})$. We have established \eqref{I34}.

The computation
\begin{align*}
\int_R (u_1 \cdots u_{s-1}Y)^{1/k-1} \d \bu' 
&=  \underset{ u_1 + \ldots + u_{s-1} < 1 }{\int_0^1 \cdots \int_0^1} 
(u_1 \cdots u_{s-1}Y)^{1/k-1} \d \bu'
\\ &= \Gamma(1/k)^s \Gamma(s/k)^{-1}
\end{align*}
is standard (see \cite[p. 22]{Dav2005}). Therefore
\[
I_4 = 2 \eta P^{-k} \Gamma(1/k)^s \Gamma(s/k)^{-1}.
\]
In view of \eqref{tweak}, \eqref{I1}, \eqref{I23} and \eqref{I34}, we now have
\[ I^* = 2 \eta \Gamma(1/k)^s \Gamma(s/k)^{-1} P^{-k} + o(P^{-k}). \]
Combining this with \eqref{d12}, \eqref{d23} and \eqref{IndDif} yields
\begin{equation} \label{finally}
\cI_\pm^{(1)} = 2 \eta \Gamma(1+1/k)^s \Gamma(s/k)^{-1} P^{s-k} + o(P^{s-k}),
\end{equation}
where we recall \eqref{I1def}. Finally, \eqref{rest3HP} and \eqref{finally} give \eqref{goal3HP}, completing the proof of Theorem \ref{asymptotic}.

\section{Classical diminishing ranges}
\label{basic}

In this section we prove Theorem \ref{basicThm}. We shall restrict some variables to lie in diminishing ranges with the exponent $1-1/k$, exploiting the fact that we may obtain square root cancellation on the even moments associated to such ranges.

\begin{lemma} \label{lazy}
Let $k \ge 2$ and $t$ be positive integers, let $h_1,\ldots,h_t \in \bR[x]$ be degree $k$ polynomials, and let $\eta$ be a positive real number. Let $c > 1$, let $\lam = 1-1/k$ and, for $j=1,2,\ldots,t$, put $\lam_j = \lam^{j-1}$. Then the number $T$ of integral solutions to
\begin{equation} \label{lazyIneq}
\Bigl| \sum_{j\le t} (h_j(x_j) - h_j(y_j)) \Bigr| < \eta
\end{equation}
with $P^{\lam_j} < x_j, y_j \le cP^{\lam_j}$ $(1 \le j \le t)$ satisfies $T \ll_\eta P^{\lam_1+\ldots+\lam_t}$.
\end{lemma}

\begin{proof} 
We proceed by induction on $t$. If $t=1$ then Lemma \ref{SecondMoment} yields $T \ll P$. Now let $t > 1$, and assume that the conclusion of Lemma \ref{lazy} holds with $t-1$ in place of $t$, for all large $P$ and all $\eta > 0$. We apply this inductive hypothesis to $h_2, \ldots, h_t$, with $P^\lam$ in place of $P$, and with $2 \eta$ in place of $\eta$. Since $\lam_j = \lam \lam_{j-1}$ ($2 \le j \le t$), this tells us the number $S$ of integer solutions $x_2,\ldots,x_t, y_2,\ldots,y_t$ to
\[ \Bigl| \sum_{j=2}^t (h_j(x_j) -h_j(y_j)) \Bigr| < 2 \eta \]
with $P^{\lambda_j} < x_j, y_j \le cP^{\lambda_j}$ ($2 \le j \le t$) satisfies 
\[ S \ll (P^\lam)^{\lambda_1+ \ldots + \lambda_{t-1}} = 
P^{\lambda_2+ \ldots + \lambda_t} . \]
Thus the number of solutions counted by $T$ with $x_1 = y_1$ is at most
\[
cPS \ll P^{\lam_1+\ldots+\lam_t}.
\]

It therefore remains to show that $T' \ll  P^{\lambda_1+ \ldots+\lambda_t}$, where $T'$ is the number of solutions counted by $T$ with $x_1 > y_1$. Put $y_1 = x$ and $x_1 = x+h$. Let $C$ be a large positive constant. The mean value theorem gives
\[ 
|h_1(x_1) - h_1(y_1)| \gg P^{k-1} |x_1 - y_1| = hP^{k-1}. 
\]
By combining this with the inequalities \eqref{lazyIneq} and
\[
\sum_{j=2}^t (h_j(x_j) -h_j(y_j)) \ll P^{k\lam_2} = P^{k-1},
\]
we deduce that $0 < h \le C$. 

Put
\[ 
f_j(\alp) = \sum_{P^{\lambda_j} < x \le cP^{\lambda_j}} e(\alp h_j(x)) 
\qquad (2 \le j \le t)
\]
and 
\[ F(\alp) = f_2(\alp) \cdots f_t(\alp).\]
For real numbers $\alp$, define
\[ G(\alp) = \sum_{0 < h \le C} \sum_{P < x \le cP} e(\alp(h_1(x+h) -h_1(x))). \]
In light of \eqref{OrthBounds}, a trivial bound on $G(\alp)$ gives
\begin{align*}
T' &\ll \int_\bR G(\alp) |F(\alp)|^2 K(2\alp) \d \alp \ll P  \int_\bR |F(\alp)|^2 K(2\alp) \d \alp  \\
& \ll PS \ll  P^{\lam_1+\ldots+\lam_t},
\end{align*}
completing the proof.
\end{proof}

We are now ready to prove Theorem \ref{basicThm}. Let $s \ge 2t+E$, and let $\eta > 0$. With $\lam = 1-1/k$, put
\begin{equation} \label{lamDefs}
\lam_j = \lam^{j-1} \quad (1 \le j \le t), \qquad \Delta = \lam_1 + \ldots + \lam_t.
\end{equation}
Let $0 < \xi < \lambda_t$, and let $\gam$ be a small positive real number. 

\subsection{A first upper bound for $s_1(k)$} \label{s1bitBasic}

In this subsection we prove \eqref{basicBounds} for $\iota =1$. Let $\tau$ be a large positive real number, and define $P$ by $\tau = (E+2.1)P^k$. We need to show that there exist integers $x_1 > \mu_1, \ldots, x_s > \mu_s$ satisfying \eqref{main}. We may plainly assume \eqref{wlog} and, by fixing the variables $x_{2t+E+1}, \ldots, x_s$ if necessary, that
\begin{equation} \label{sdef} 
s = 2t + E.
\end{equation} 

Let
\[
g_i(\alp) = \sum_{P^{\lambda_j} < x \le 2P^{\lambda_j}} e(\alp (x- \mu_i)^k) \qquad (1 \le i \le s),
\]
where 
\begin{equation} \label{jdef}
j = j(i) =  \begin{cases}
1, & 1 \le i \le  E \\
\lfloor (i - E+1)/2 \rfloor, & E < i \le s.
\end{cases} 
\end{equation}
By \eqref{OrthBounds}, it suffices to prove that
\[ \int_\bR g_1(\alp) \cdots g_s(\alp) e(-\alp \tau) K(\alp) \d \alp \gg P^{E+2\Delta-k}. \]
Recall that $\mu_1 \notin \bQ$. With $T(P)$ as in Lemma \ref{Freeman}, applied to the polynomials $(x-\mu_1)^k$ and $(x- \mu_2)^k$, we define our Davenport-Heilbronn arcs by \eqref{dh1}, \eqref{dh2} and \eqref{dh3}.

\begin{lemma} \label{s1basicMajor}
We have 
\[ \int_\fM g_1(\alp) \cdots g_s(\alp) e(-\alp \tau) K(\alp) \d \alp \gg P^{E+2\Delta-k}. \]
\end{lemma}

\begin{proof}
For $X > 0$ and $\alp \in \bR$ write
\[
I(\alp, X) = \int_X^{2X} e(\alp x^k) \d x.
\]
With $j = j(i)$ as in \eqref{jdef}, let
\[ 
I_i(\alp) = I(\alp, P^{\lam_j}) \qquad (1 \le i \le s).
\]
Define 
\begin{align*} 
\cI^{(1)} &= \int_\fM g_1(\alp)\cdots g_s(\alp)e(-\alp\tau) K(\alp) \d \alp, \\
\cI^{(2)} &= \int_\fM I_1(\alp)\cdots I_s(\alp)e(-\alp\tau) K(\alp) \d \alp 
\end{align*}
and
\[ \cI^{(3)} = \int_\bR I_1(\alp)\cdots I_s(\alp)e(-\alp\tau) K(\alp) \d \alp. \]

Given $\mu \in \bR$ and $X \in (0,P]$, it follows from \cite[Lemma 4.4]{Bak1986} that if $\alp \in \fM$ then
\begin{align*}
\sum_{X < x \le 2X} e(\alp(x-\mu)^k) &= \int_X^{2X} e(\alp(x-\mu)^k)\d x+ O(1) \\
&= I(\alp,X)+O(1).
\end{align*}
Recalling \eqref{Kbound}, we now conclude that
\begin{equation} \label{d1}
\cI^{(1)} - \cI^{(2)} \ll 
P^{\xi-k}P^{E+2\Delta-\lam_t} = o(P^{E+2\Delta-k}),
\end{equation}
since $\xi < \lam_t$. By \cite[Theorem 7.3]{Vau1997} we have
\[ 
I(\alp,P) \ll |\alp|^{-1/k}, 
\]
and now \eqref{Kbound} and a trivial estimate give
\begin{equation} \label{d2} \cI^{(3)} - \cI^{(2)} \ll  P^{2\Delta-2} \int_{P^{\xi-k}}^\infty \alp^{-(E+2)/k} \d \alp
= o(P^{E+2\Delta-k}),
\end{equation}
as $E \ge 2k - 1 > k-2$.

With $j = j(i)$ as in \eqref{jdef}, write $\cR_0 = \prod_{i \le s} (P^{\lam_j}, 2P^{\lam_j}]$, and consider
\[ 
\cI^{(3)} = \int_{\cR_0} \int_\bR e(\alp(x_1^k + \ldots + x_s^k-\tau)) K(\alp) \d \alp \d \bx.
\]
By \eqref{orth1}, changing variables yields
\begin{equation} \label{I3first}
\cI^{(3)} \gg \int_\cR (\eta - |y_1 + \ldots + y_s - \tau|) \cdot (y_1 \cdots y_s)^{1/k-1} \d \by,
\end{equation}
where $\cR$ is the set of $\by \in \bR^s$ such that
\begin{equation} \label{ybounds}
P^{k\lam_{j(i)}} < y_i \le 2^k P^{k\lam_{j(i)}} \qquad (1 \le i \le s) 
\end{equation}
and
\[ 
|y_1 + \ldots + y_s - \tau| < \eta. 
\]
Let $\omega$ be a small positive real number. Let $\cV$ denote the set of $\by \in \cR$ such that
\begin{equation} \label{restrict}
P^k < y_2, y_3, \ldots, y_{E+2} \le (1+\omega)P^k
\end{equation}
and
\[ 
|y_1 + \ldots + y_s - \tau| < \eta/2. 
\]
By positivity of the integrand in \eqref{I3first}, we have
\[
\cI^{(3)} \gg P^{k(E+2\Delta)(1/k-1)} \cdot \meas(\cV) = P^{(1-k)(E+2\Delta)} \cdot \meas(\cV).
\]

As $\tau = (E+2.1) P^k$ and $\omega$ is small, we have
\[ 
P^k +\eta < \tau - y_2 - \ldots - y_s < 1.1P^k
\]
whenever the inequalities \eqref{ybounds} and \eqref{restrict} are satisfied. Hence
\[
\meas(\cV) \gg P^{k(E-1+2\Delta)},
\]
so
\begin{equation} \label{I3}
\cI^{(3)} \gg  P^{(1-k)(E+2\Delta)} P^{k(E-1+2\Delta)} = P^{E+2\Delta-k}.
\end{equation}
The bounds \eqref{d1}, \eqref{d2} and \eqref{I3} yield the desired result 
\[ \cI^{(1)} \gg P^{E+2\Delta-k}. \]
\end{proof}

By Lemma \ref{s1basicMajor}, H\"older's inequality and symmetry, it remains to show that 
\begin{equation} \label{basicGoal}
\int_{\fm \cup \ft} |g_1(\alp) g_2(\alp) g_3(\alp)^E \prod_{2 \le j \le t} g_{E+2j-1}(\alp)^2| K(\alp) \d \alp
= o(P^{E+2\Delta-k}).
\end{equation}
By inspecting \eqref{jdef}, we see that 
\[
g_{E+2j-1}(\alp) = \sum_{P^{\lam_j} < x \le 2P^{\lam_j}}e(\alp(x-\mu_{E+2j-1})^k).
\]
Fix $i \in \{1,2,3\}$, let
\[ \fN = \{ \alp \in \bR: |g_i(\alp)| > P^{1-\sigma(k)+\eps} \}, \]
put $\fn = \bR \setminus \fN$, and let $\fU$ be the intersection of $\fN$ with a unit interval. For subsets $U \subseteq \bR$, write
\[ I(U) = \int_U |g_i(\alp)^{E+2-\gam} \prod_{2 \le j \le t} g_{E+2j-1}(\alp)^2| K(\alp) \d \alp.\]
In view of \eqref{OrthBounds} and \eqref{lamDefs}, Lemma \ref{lazy} implies that
\[ \int_\bR |g_i(\alp)^2 \prod_{2 \le j \le t} g_{E+2j-1}(\alp)^2| K(\alp) \d \alp \ll P^\Delta. \]
A straightforward computation gives $\Delta = k(1-\lam^t)$, so 
\[ E > k \lam^t \sigma(k)^{-1} =  (k-\Delta) \sigma(k)^{-1}. \]
As $\gam$ and $\eps$ are small, we must therefore have 
\[ (E-\gam)\sig(k) - \eps > k - \Delta. \]
Hence
\begin{align} \notag
I(\fn) 
&\ll  (\sup_{\alp \in \fn} |g_i(\alp)|)^{E - \gam} P^\Delta \le P^{(E-\gam)(1-\sigma(k))+\Delta + \eps} \\
\label{basicClassicalMinor} & = o(P^{E+2\Delta-k-\gam}).
\end{align}

Since $E+2 \ge 2k+1$, Corollary \ref{ClassicalMajorGeneral} and a trivial estimate yield
\[
\int_\fU |g_i(\alp)^{E+2-\gam} \prod_{2 \le j \le t} g_{E+2j-1}(\alp)^2| \d \alp 
\ll P^{E+2\Delta-\gam-k}
\]
which, recalling \eqref{Kbound}, gives
\begin{equation} \label{basicClassicalMajor}
I(\fN) \ll P^{E+2\Delta-k-\gam}
\end{equation}
and
\begin{align} \label{basicTrivialMajor}
I(\ft \cap \fN) \ll P^{E+2\Delta-k-\gam} \sum_{n=0}^\infty (T(P)+n)^{-2} 
= o(P^{E+2\Delta-k-\gam}).
\end{align}
The inequalities \eqref{basicClassicalMinor} and \eqref{basicTrivialMajor}, together with a trivial estimate and H\"older's inequality, yield
\begin{equation} \label{basicTrivial}
\int_\ft |g_1(\alp)g_2(\alp) g_3(\alp)^E \prod_{2 \le j \le t} g_{E+2j-1}(\alp)^2| K(\alp) \d \alp = o(P^{E+2\Delta-k}).
\end{equation}

Combining \eqref{basicClassicalMinor} with \eqref{basicClassicalMajor} gives
\[
\int_\bR |g_i(\alp)^{E+2-\gam} \prod_{2 \le j \le t} g_{E+2j-1}(\alp)^2| K(\alp) \d \alp \ll P^{E+2\Delta-k-\gam}
\]
for $i=1,2,3$ which, by H\"older's inequality and \eqref{Freeman0}, yields
\[
\int_\fm |g_1(\alp)g_2(\alp) g_3(\alp)^E  \prod_{2 \le j \le t} g_{E+2j-1}(\alp)^2| K(\alp) \d \alp 
= o(P^{E+2\Delta-k}).
\]
This and \eqref{basicTrivial} give \eqref{basicGoal}, completing the proof of \eqref{basicBounds} for $\iota=1$.

\subsection{An upper bound for $s_0(k)$}

In this subsection we prove \eqref{basicBounds} for \mbox{$\iota =0$.} Let $h_1, \ldots, h_s \in \bR[x]$ be degree $k$ polynomials satisfying the irrationality condition, and put $H(\bx) = \sum_{i \le s} h_i(x_i)$. Let $\tau \in \bR$, and assume that $H(\bx)$ is indefinite. We need to show that there exists $\bx \in \bZ^s$ satisfying \eqref{main2}. Without loss of generality $h_1$ and $h_2$ satisfy the irrationality condition. We may evidently assume that $\tau = 0$, and that $h_1$ is monic. Let $a_1=1, a_2, \ldots, a_s$ be the leading coefficients of $h_1, \ldots, h_s$ respectively. As $H(\bx)$ is indefinite, we may assume without loss that $a_J < 0$ for some $J \in\{ 2, 3 \}$. By fixing the variables $x_{2t+E+1}, \ldots, x_s$ if necessary, we may plainly assume \eqref{sdef}.

Define $r \in \bR$ by
\begin{equation} \label{rdef}
- a_J r^k = 3 + \sum_{i \le s} |a_i|,
\end{equation}
and note that $r > 1$. Let $\omega$ be a small positive constant, and let $c$ be a large positive constant. Let $P$ be a large positive real number. With $j = j(i)$ as in \eqref{jdef}, let
\[
g_i(\alp) = \sum_{P^{\lambda_j} < x \le cP^{\lambda_j}} e(\alp h_i(x)) \qquad (1 \le i \le s),
\]
where we recall \eqref{lamDefs}. By \eqref{OrthBounds}, it suffices to prove that
\[ 
\int_\bR g_1(\alp) \cdots g_s(\alp) K(\alp) \d \alp \gg P^{E+2\Delta-k}. 
\]
With $T(P)$ as in Lemma \ref{Freeman}, we define our Davenport-Heilbronn arcs by \eqref{dh1}, \eqref{dh2} and \eqref{dh3}.

\begin{lemma} \label{s0basicMajor} We have
\[ \int_\fM g_1(\alp) \cdots g_s(\alp) K(\alp) \d \alp \gg P^{E+2\Delta-k}. \]
\end{lemma}

\begin{proof} Let
\[ I_i(\alp) = \int_{P^{\lam_j}}^{cP^{\lam_j}} e(\alp h_i(x)) \d x \qquad (1 \le i \le s),\]
with $j = j(i)$ as in \eqref{jdef}. Define 
\begin{align*} 
\cI^{(1)} &= \int_\fM g_1(\alp)\cdots g_s(\alp) K(\alp) \d \alp, \\
\cI^{(2)} &= \int_\fM I_1(\alp)\cdots I_s(\alp) K(\alp) \d \alp 
\end{align*}
and
\[ \cI^{(3)} = \int_\bR I_1(\alp)\cdots I_s(\alp) K(\alp) \d \alp. \]
Mimicking the proofs of \eqref{d1} and \eqref{d2}, we deduce that
\begin{equation} \label{s0d12}
\cI^{(1)} - \cI^{(2)} = o(P^{E+2\Delta-k}) 
\end{equation}
and 
\begin{equation} \label{s0d23}
\cI^{(2)} - \cI^{(3)} = o(P^{E+2\Delta-k}). 
\end{equation}

With $j = j(i)$ as in \eqref{jdef}, write $R = \prod_{i \le s} (P^{\lam_j}, cP^{\lam_j}]$, and consider
\[ 
\cI^{(3)} = \int_R \int_\bR e(\alp H(\bx)) K(\alp) \d \alp \d \bx.
\]
Let $I = \{2, 3, \ldots, E+2 \} \setminus \{ J \}$. As $r > 1$ and $\omega$ is small, we must also have $r-\omega > 1$. Let $X$ denote the set of $(x_2, \ldots, x_s) \in \bR^{s-1}$ such that
\begin{align*}
(r- \omega)P &\le x_J \le (r+\omega)P, \\
P &\le x_i \le (1+\omega)P \qquad (i \in I) 
\end{align*}
and 
\[
P^{\lam_{j(i)}} < x_i \le cP^{\lam_{j(i)}} \qquad (E+2 < i \le s).
\]
By \eqref{orth1}, we have
\begin{equation} \label{spec}
\cI^{(3)} \gg \int_R \max(0, \eta - |H(\bx)|) \d \bx \gg V,
\end{equation}
where $V$ is the measure of the set of $\bx \in [P,cP] \times X$ such that $|H(\bx)| \le \eta/2$. In view of \eqref{s0d12}, \eqref{s0d23} and \eqref{spec}, it remains to show that
\begin{equation} \label{unif}
 \meas \{x_1 \in [P,cP] : |H(\bx)| \le \eta/2 \} \gg P^{1-k}, 
\end{equation}
uniformly for $(x_2, \ldots, x_s) \in X$.

Let $\bx' = (x_2, \ldots, x_s) \in X$, and put
\[ 
\Lambda (\bx') =  - \eta/2 - \sum_{i =2}^s h_i(x_i). 
\]
Then 
\[ 
\cL P^k < \Lambda (\bx') < \cU P^k,
\]
where
\[
\cL = -a_J(r-2\omega)^k - (1+2\omega)^k \sum_{i \in I} |a_i|
\]
and
\[
\cU = -a_J(r+2\omega)^k + (1+2\omega)^k \sum_{i \in I} |a_i|.
\]
Since $\omega$ is small, it follows from \eqref{rdef} that $\cL >2$. As $c$ is large, we also have $\cU < c^k - 1$. Now
\begin{equation} \label{Lambound}
2 P^k < \Lambda (\bx') < (c^k - 1) P^k.
\end{equation}

The polynomial $h_1$ is strictly increasing when its argument is sufficiently large. As $\Lambda(\bx')$ is large and positive, there exist unique positive real numbers $m$ and $M$ such that $h_1(m) = \Lambda(\bx')$ and $h_1(M) = \Lambda(\bx') + \eta$. Recalling that $h_1$ is monic of degree $k$, we now deduce from \eqref{Lambound} that
\[ 
P < m < M < cP.
\]
Any $x_1 \in [m,M]$ satisfies $|H(\bx)| \le \eta/2$, and the mean value theorem gives
\[ (M-m)^{-1} \ll (M-m)^{-1} (h_1(M) - h_1(m)) \ll M^{k-1} \ll P^{k-1}. \]
Thus we have \eqref{unif}, completing the proof of Lemma \ref{s0basicMajor}.
\end{proof}

The remainder of the proof is identical to that of \eqref{basicBounds} for $\iota=1$, which is given in \S \ref{s1bitBasic}. Thus we have established \eqref{basicBounds} for $\iota=0$, thereby completing the proof of Theorem \ref{basicThm}.

\section{Slowly diminishing ranges}
\label{mainProof}

In this section we deduce the bounds in Table \ref{table1} by proving a more general result. Recalling \eqref{sigmaDef}, we henceforth write 
\begin{equation} \label{henceforth}
v = \sigma(k-1), \qquad \lam = 1 - (1-v)/k.
\end{equation}
We shall use $\lam$ as our diminishing ranges exponent. Using an iterative procedure, we achieve square root cancellation on low even moments associated to these diminishing ranges. The $t$th step fails by a power of $E^*$ to deliver square root cancellation on the $2t$th moment; see Lemma \ref{moment}. Definition \ref{CrazyDef} coins an adjective for when $E^*$ vanishes at each of the first $n$ steps.
\begin{defn} \label{CrazyDef}
Let $k \ge 2$. An integer $n$ is \emph{$k$-good} if $n=0$, or if $n > 0$ and the following holds for $t=1,2,\ldots,n$. With \eqref{lamDefs}, let
\begin{equation} \label{e1def}
 e_1 = v + 1/2 + (\Delta-1)(2k-1-2\sig(k))/(2k-2)- \Delta
\end{equation}
and, for $2 \le \ell \le t$, let
\begin{equation} \label{eldef}
e_\ell = v + \Delta-2-2\sig(k)\cdot (\lam_{\ell+1} + \ldots + \lam_t)+(k-2)(v-1) + M_\ell,
\end{equation}
where $M_\ell = \max(0, 1+2(1- \ell)/k) \cdot \lam_\ell k\sig(k)$. Put
\begin{equation}\label{EstarDef}
E^* = \max(0, e_1, \ldots, e_t).
\end{equation}
Then $E^* = 0$.
\end{defn}

Note that if $n \in \bN$ is $k$-good then so is $n-1$. We shall prove the following.

\begin{thm} \label{mainThm}
Let $k \ge 4$. If $k=4$, let $t = 4$ and $E = 8$. Otherwise, let $t$ be a $k$-good integer such that
\begin{equation} \label{cond1}
\lam^{t-1} > \frac{v}{1-\sig(k)}
\end{equation}
and
\begin{equation} \label{cond2}
2t+E \ge 4k,
\end{equation}
where
\begin{equation} \label{Edef}
E = 1 + \max(k-1, \lfloor \sig(k)^{-1} k (\lam^t -v) /(1-v) \rfloor).
\end{equation}
Then
\[
s_1(k) \le 2t + E.
\]
\end{thm}

Choosing $t$ optimally yields the bounds implicit in Table \ref{table1}. It therefore remains to prove Theorem \ref{mainThm}. Note that 3 is 4-good but 4 is not. Note that \eqref{cond1} and \eqref{cond2} also hold in the case $k=4$. We begin by establishing some even moment estimates. 

\begin{lemma} \label{moment}
Let $k \ge 4$, let $t$ be a positive integer satisfying \eqref{cond1}, and assume that $t-1$ is $k$-good. Recall \eqref{lamDefs}, \eqref{henceforth} and \eqref{EstarDef}. Let $\bmu \in \bR^t$, and let $\eta > 0$. Then the number $T$ of integral solutions to
\begin{equation} \label{momentIneq}
\Bigl| \sum_{j\le t} ((x_j - \mu_j)^k - (y_j-\mu_j)^k) \Bigr| < \eta 
\end{equation}
with $P^{\lam_j} < x_j, y_j \le 2P^{\lam_j}$ $(1 \le j \le t)$ satisfies $T \ll_\eta P^{\Delta+\eps+E^*}$.
\end{lemma}

\begin{proof} 
We proceed by induction on $t$. By Lemma \ref{SecondMoment}, the conclusion holds for $t=1$. Let $t>1$, and assume that the conclusion of Lemma \ref{moment} holds with $t-1$ in place of $t$, for all large $P$ and all $\eta > 0$. We shall apply this inductive hypothesis with $P^\lam$ in place of $P$, and with $2 \eta$ in place of $\eta$. To do so, we note firstly that $t-2$ is $k$-good because $t-1$ is, and secondly that \eqref{cond1} implies that $\lam^{t-2} > v/(1-\sig(k))$. Since $t-1$ is $k$-good, and since $\lam_j = \lam \lam_{j-1}$ ($2 \le j \le t$), the inductive hypothesis tells us that the number $S$ of integer solutions $x_2,\ldots,x_t, y_2,\ldots,y_t$ to
\begin{equation} \label{Sdef}
\Bigl| \sum_{j=2}^t ((x_j - \mu_j)^k - (y_j-\mu_j)^k) \Bigr|  < 2 \eta 
\end{equation}
with $P^{\lambda_j} < x_j, y_j \le 2 P^{\lambda_j}$ ($2 \le j \le t$) satisfies 
\begin{equation} \label{Sbound}
S \ll (P^\lam)^{\lam_1 + \ldots + \lam_{t-1}+\eps}
\le P^{\lambda_2+ \ldots + \lambda_t+\eps} 
= P^{\Delta-1 + \eps}. 
\end{equation}
Thus the number of solutions counted by $T$ with $x_1 = y_1$ is at most 
\[
(P+1) S \ll P^{\Delta+\eps}.
\]

Since $E^* \ge 0$, it therefore remains to show that $T' \ll  P^{\Delta+\eps+E^*}$, where $T'$ is the number of solutions counted by $T$ with $x_1 > y_1$. Put $y_1 = x$ and $x_1 = x+h$. Let $C$ be a large positive constant, and write $H = CP^v$. The mean value theorem gives
\[
|(x_1-\mu_1)^k - (y_1 - \mu_1)^k| \gg P^{k-1} |x_1 - y_1| = hP^{k-1}.
\]
By combining this with the inequalities \eqref{momentIneq} and
\[
\sum_{j=2}^t ((x_j - \mu_j)^k - (y_j-\mu_j)^k) \ll P^{k \lam_2} = P^{k \lam},
\]
we deduce that
\[ 
0 < h \le CP^{k\lam - k + 1} = CP^v = H. 
\]

Put
\[ f_j(\alp) = \sum_{P^{\lambda_j} < x \le 2P^{\lambda_j}} e(\alp (x- \mu_j)^k)  \qquad (2 \le j \le t) \]
and 
\[ F(\alp) = f_2(\alp) \cdots f_t(\alp).\]
For integers $h$ and real numbers $\alp$, define
\[
\Phi_h(\alp) = \sum_{P < x \le 2P} e(\alp(x+h-\mu_1)^k - \alp(x-\mu_1)^k).
\]
By \eqref{OrthBounds}, we have
\begin{align} \notag
T' &\ll \int_\bR \sum_{h \le H} \Phi_h(\alp) |F(\alp)|^2 K(2\alp) \d \alp \\
\label{firstT} &\le \sum_{h \le H} \int_\bR |\Phi_h(\alp) F(\alp)^2| K(2\alp) \d \alp.
\end{align}

Let $h \in \bN$ and $\alp \in \bR$. The polynomial associated to the Weyl sum $\Phi_h(\alp)$, namely
\[ (x+h-\mu_1)^k - (x-\mu_1)^k, \]
has degree $k-1$ and leading coefficient $kh$. Thus we may apply Lemma \ref{ClassicalMajorIngredient} to the polynomial
\[ \frac{(x+h-\mu_1)^k - (x-\mu_1)^k}{kh},\]
with $d=k-1$, and with $kh\alp$ in place of $\alp$. We thereby deduce that if
\[ |\Phi_h(\alp)| > P^{1-v+\eps} \]
then there exist relatively prime integers $a$ and $q$ such that
\begin{align}
\label{mainqbound} 0 < q &< P^{(k-1)v}, \\
\label{mainbbound} |qkh\alp - a| &< P^{(k-1)(v-1)}
\end{align}
and
\begin{equation} \label{mainPhibound} 
\Phi_h(\alp) \ll q^{\eps - 1/(k-1)}P (1+P^{k-1} kh |\bet|)^{-1/(k-1)},
\end{equation}
where $\bet = \alp - a/(qkh)$. 

For $h, q \in \bN$ and $a \in \bZ$, denote by $\fM_h(q,a)$ the set of $\alp \in \bR$ satisfying \eqref{mainbbound} and \eqref{mainPhibound}. For $h, q \in \bN$, let $\fM_h(q)$ denote the union of the sets $\fM_h(q,a)$ over integers $a$ such that $(a,q) = 1$. For $h \in \bN$, we make the following definitions. Let
\[
\fm_h = \{ \alp \in \bR :  |\Phi_h(\alp)| \le P^{1-v+\eps} \}.
\]
By \eqref{sigmaDef} and \eqref{henceforth}, it is easy to show that $k\sig(k) < (k-1)v$. Denote by $\fM_h^{(1)}$ the union of the sets $\fM_h(q)$ over integers $q$ such that
\begin{equation} \label{qrange1}
P^{\lam_2 k \sig(k)} \le q < P^{(k-1)v}. 
\end{equation}
Write $\lam_{t+1} = 0$. For $\ell =2,3,\ldots,t$, denote by $\fM_h^{(\ell)}$ the union of the sets $\fM_h(q)$ over integers $q$ such that
\begin{equation} \label{qrange2}
P^{\lam_{\ell+1} k \sig(k)} \le q < P^{\lam_\ell k \sig(k)},
\end{equation}
and let $\fU_h^{(\ell)}$ be the intersection of $\fM_h^{(\ell)}$ with a unit interval. From the discussion above, we have
\[
\bR = \fm_h \cup \Bigl( \bigcup_{\ell \le t} \fM_h^{(\ell)} \Bigr). 
\]

Now \eqref{firstT} gives
\[
T' \ll I_\fm + \sum_{h \le H} \sum_{\ell \le t}  I_h^{(\ell)}, 
\]
where
\[
 I_\fm = \sum_{h \le H} \int_{\fm_h} |\Phi_h(\alp) F(\alp)^2| K(2\alp) \d \alp
\]
and
\[
 I_h^{(\ell)} = \int_{\fM_h^{(\ell)}} |\Phi_h(\alp) F(\alp)^2| K(2\alp) \d \alp
\qquad (h \le H, 1 \le \ell \le t).
\]
Moreover, by \eqref{OrthBounds} and \eqref{Sbound}, we have
\[
I_\fm \le \sum_{h \le H} (\sup_{\alp \in \fm_h}|\Phi_h(\alp)|) \cdot S
\le P^{1-v+\eps}HS \ll P^{\Delta+2\eps}.
\]
Since $H \ll P^v$ and $E^* \ge 0$, it now suffices to prove that
\begin{equation} \label{hgoal}
 I_h^{(\ell)} \ll P^{\Delta + E^* - v + \eps}
\qquad (1 \le \ell \le t),
\end{equation}
uniformly in positive integers $h \le H$. Let $h \le H$ be a positive integer.

Now we show that if $q \in \bN$, $q < P^{(k-1)v}$, $\alp \in \fM_h(q)$, $2 \le j \le t$ and
\begin{equation} \label{hyp}
|f_j(\alp)| > P^{\lam_j(1-\sig(k))+\eps}
\end{equation}
then 
\begin{equation} \label{concs}
q  < P^{\lam_j k  \sig(k)}, \qquad f_j(\alp) \ll q^{\eps-1/k}P^{\lam_j}.
\end{equation}
Let $q \in \bN$, $\alp \in \fM_h(q)$ and $2 \le j \le t$, and assume \eqref{mainqbound} and \eqref{hyp}. As $\alp \in \fM_h(q)$, there exists $a \in \bZ$ such that $(a,q)=1$ and $\alp \in \fM_h(q,a)$. By Lemma \ref{ClassicalMajorIngredient}, there exist integers $u$ and $r$ such that
\begin{equation} \label{rbounds}
0 < r < P^{\lam_j k \sig(k)}, \qquad |r\alp - u| < P^{\lam_j k(\sig(k)-1)} 
\end{equation}
and
\begin{equation} \label{r_est}
f_j (\alp) \ll r^{\eps-1/k} P^{\lam_j}.
\end{equation}
By \eqref{mainbbound} and \eqref{rbounds} we have
\begin{align}
\notag |a/q-khu/r| &\le |kh\alp-a/q| + |kh\alp - khu/r| \\
\label{fracdif1} &< \frac1{qP^{(k-1)(1-v)}} + \frac{kh}{rP^{\lam_j k(1-\sig(k))}}.
\end{align}
As $\lam_j \ge \lam^{t-1}$, we note from \eqref{cond1} that $\lam_j > v/(1-\sig(k))$. Since $h \le H = CP^v$, we can now use \eqref{mainqbound} and \eqref{rbounds} to deduce from \eqref{fracdif1} that $|a/q - khu/r| < (qr)^{-1}$. Hence $a/q = khu/r$. Thus, recalling that $(a,q)=1$, we see that $q \le r$. Now \eqref{rbounds} and \eqref{r_est} imply \eqref{concs}.

We draw the following conclusions from the above discussion. Firstly, if $\alp \in \fM_h^{(1)}$ then $\alp \in \fM_h(q)$ for some integer $q$ satisfying $P^{\lam_j k\sig(k)} \le q < P^{1-v}$ ($2 \le j \le t$), so
\begin{equation} \label{Fbound}
|F(\alp)| \le P^{(\Delta-1)(1-\sig(k))+\eps}.
\end{equation}
Secondly, if $2 \le \ell \le t$ and $\alp \in \fM_h(q)$ for some integer $q$ satisfying \eqref{qrange2} then
\begin{equation} \label{ded1}
|f_j(\alp)| \le P^{\lam_j(1-\sig(k))+\eps} \qquad (\ell < j \le t)
\end{equation}
and
\begin{equation} \label{ded2}
f_j(\alp) \ll P^{\lam_j(1-\sig(k))+\eps} + q^{\eps-1/k}P^{\lam_j} \ll q^{\eps' - 1/k}P^{\lam_j} \qquad (2 \le j \le \ell),
\end{equation}
where $\eps' = (\lam_j k\sig(k))^{-1}\eps$. 

H\"older's inequality gives
\begin{equation} \label{Ih1decomp}
I_h^{(1)} \le J_{1,h}^{1/(2k-2)} J_{2,h}^{(2k-3)/(2k-2)},
\end{equation}
where
\[
J_{1,h} = \int_{\fM_h^{(1)}} |\Phi_h(\alp)|^{2k-2} K(2\alp) \d \alp
\]
and
\[
J_{2,h} = \int_{\fM_h^{(1)}} |F(\alp)|^{(4k-4)/(2k-3)} K(2\alp) \d \alp.
\]
If $q \in \bN$ then there are at most $qkh+1$ integers $a$ satisfying \eqref{mainbbound} for some $\alp \in \fU_h^{(1)}$. Thus, by \eqref{mainPhibound} and \eqref{qrange1}, we have
\[
\int_{\fU_h^{(1)}} |\Phi_h(\alp)|^{2k-2} \d \alp 
\ll \sum_{q < P^{(k-1)v}} qh \cdot q^{\eps-2} P^{2k-2}  \fJ_h,
\]
where 
\[ 
\fJ_h = \int_0^\infty (1+P^{k-1} kh \bet)^{-2} \d \bet
\ll h^{-1} P^{1-k}.
\]
Hence
\[
\int_{\fU_h^{(1)}} |\Phi_h(\alp)|^{2k-2} \d \alp 
\ll P^{k-1+\eps}
\]
which, by \eqref{Kbound}, yields 
\begin{equation} \label{J1h}
J_{1,h} \ll P^{k-1+\eps}.
\end{equation}
Moreover, on recalling the definition of $S$ from \eqref{Sdef}, the inequalities \eqref{OrthBounds}, \eqref{Sbound} and \eqref{Fbound} give
\begin{align} \notag
J_{2,h} &\ll \sup_{\alp \in \fM_h^{(1)}} |F(\alp)| ^{2/(2k-3)} S
\ll P^{(\Delta-1)(1-\sig(k))2/(2k-3)+\eps}S 
\\  
\label{J2h}
& \ll P^{(\Delta-1)(2k-1-2\sig(k))/(2k-3)+2\eps}.
\end{align}
Substituting \eqref{J1h} and \eqref{J2h} into \eqref{Ih1decomp} yields
\begin{equation} \label{Ih1}
I_h^{(1)} \ll P^{1/2 + (\Delta-1)(2k-1-2\sig(k))/(2k-2)+\eps} = P^{\Delta+e_1 - v +\eps},
\end{equation}
where we recall \eqref{e1def}.

Let $\ell \in \{2,3,\ldots,t\}$. If $q \in \bN$ then there are at most $qkh+1$ integers $a$ satisfying \eqref{mainbbound} for some $\alp \in \fU_h^{(\ell)}$. Now \eqref{mainbbound}, \eqref{mainPhibound}, \eqref{qrange2}, \eqref{ded1} and \eqref{ded2} yield
\[
I_h^{(\ell)} \ll  P^{1+2(\Delta-1)-2\sig(k)\cdot (\lam_{\ell+1} + \ldots + \lam_t)+\eps}   X_h^{(\ell)},
\]
where
\[ 
X_h^{(\ell)} = \sum_{q < P^{\lam_\ell k\sig(k)}}  qh \cdot q^{2(1-\ell)/k-1/(k-1)} \cJ_{h,q},
\]
where
\[
\cJ_{h,q} = \int_0^{(qkh)^{-1}P^{(k-1)(v-1)}} (1+P^{k-1}kh\bet)^{-1/(k-1)} \d \bet.
\]
The calculation
\begin{align*}
\cJ_{h,q} &\ll P^{-1}h^{-1/(k-1)} \cdot ((qkh)^{-1}P^{(k-1)(v-1)})^{(k-2)/(k-1)} \\
&\ll h^{-1}q^{-(k-2)/(k-1)}P^{-1+(k-2)(v-1)}
\end{align*}
now gives
\begin{align} 
\notag I_h^{(\ell)} &\ll P^{2\Delta-2-2\sig(k)\cdot (\lam_{\ell+1} + \ldots + \lam_t)+(k-2)(v-1)+\eps}  
\sum_{q < P^{\lam_\ell k\sig(k)}} q^{2(1- \ell)/k}\\
\label{Ihl} &\ll P^{\Delta+e_\ell-v + 2\eps} \qquad (2 \le \ell \le t),
\end{align}
where we recall \eqref{eldef}.

In light of \eqref{EstarDef}, the bounds \eqref{Ih1} and \eqref{Ihl} yield \eqref{hgoal}, completing the proof of Lemma \ref{moment}.
\end{proof}

We are ready to prove Theorem \ref{mainThm}. Let $s \ge 2t+E$, and let $\eta > 0$. Let $\tau$ be a large positive real number, and define $P$ by $\tau = (E + 2.1)P^k$. We need to show that there exist integers $x_1 > \mu_1, \ldots, x_s > \mu_s$ satisfying \eqref{main}. By fixing the variables $x_{2t+E+1}, \ldots, x_s$ if necessary, we may plainly assume \eqref{wlog} and \eqref{sdef}. Recall \eqref{lamDefs} and \eqref{henceforth}. Let $0 < \xi < \lam_t$, and let $\gam$ be a small positive real number. Define $g_1, \ldots, g_s$, $T(P)$ and our Davenport-Heilbronn arcs as in \S \ref{s1bitBasic}, but note that $\lam$ is different here. By \eqref{OrthBounds}, it suffices to prove that
\[ \int_\bR g_1(\alp) \cdots g_s(\alp) e(-\alp \tau) K(\alp) \gg P^{E+2\Delta-k}. \]

Noting that $E \ge k > k-2$, the proof of Lemma \ref{s1basicMajor} gives
\[ 
\int_\fM g_1(\alp) \cdots g_s(\alp) e(-\alp \tau) K(\alp) \d \alp \gg P^{E+2\Delta-k}. 
\]
By H\"older's inequality and symmetry, it now remains to show that 
\[
\int_{\fm \cup \ft} |g_1(\alp) g_2(\alp) g_3(\alp)^E \prod_{2 \le j \le t} g_{E+2j-1}(\alp)^2| K(\alp) \d \alp
= o(P^{E + 2 \Delta - k}).
\]
By inspecting \eqref{jdef}, we see that 
\[
g_{E+2j-1}(\alp) = \sum_{P^{\lam_j} < x \le 2P^{\lam_j}}e(\alp(x-\mu_{E+2j-1})^k).
\]
Fix $i \in \{1,2,3\}$, let
\[ \fN = \{ \alp \in \bR: |g_i(\alp)| > P^{1-\sigma(k)+\eps} \}, \]
put $\fn = \bR \setminus \fN$, and let $\fU$ be the intersection of $\fN$ with a unit interval.  For subsets $U \subseteq \bR$, write
\[ I(U) = \int_U |g_i(\alp)^{E+2-\gam} \prod_{2 \le j \le t} g_{E+2j-1}(\alp)^2| K(\alp) \d \alp.\]

If $k = 4$ then \eqref{OrthBounds} and Lemma \ref{moment} yield
\[ \int_\bR |g_i(\alp)^2 \prod_{2 \le j \le t} g_{E+2j-1}(\alp)^2| K(\alp) \d \alp \ll P^{3.011}, \]
since 3 is 4-good and $(13/16)^3 > 2/7$. Now
\[
I(\fn) \ll (\sup_{\alp \in \fn} |g_i(\alp)|)^{8-\gam} \cdot P^{3.011}
\le  P^{3.011+ (8-\gam)7/8 + \eps}.
\]
Thus, as $\gam$ is small, we have
\[
I(\fn) \ll P^{10.011 - 7\gam / 8 +\eps} = o(P^{E+2\Delta - k - \gam})
\]
when $k=4$. The exponents have been computed by machine. 

If $k \ge 5$ then $t$ is $k$-good, so \eqref{OrthBounds} and Lemma \ref{moment} yield
\[ \int_\bR |g_i(\alp)^2 \prod_{2 \le j \le t} g_{E+2j-1}(\alp)^2| K(\alp) \d \alp \ll P^{\Delta+\eps}. \]
Now
\[
I(\fn) \ll (\sup_{\alp \in \fn} |g_i(\alp)|)^{E-\gam} \cdot P^{\Del+\eps} 
\le  P^{\Delta + (E-\gam)(1-\sig(k)) + 2\eps}. 
\]
Recalling \eqref{lamDefs}, \eqref{henceforth} and \eqref{Edef}, we deduce that
\[ 
E\sig(k) > k(\lam^t-v)/(1-v) = k - \Delta. 
\]
As $\gam$ and $\eps$ are small, we must therefore have
\[
(E-\gam)\sig(k) - 2\eps > k - \Delta.
\]
Hence
\begin{equation*} 
I(\fn) = o(P^{E+2\Delta- k - \gam}) 
\end{equation*}
for $k \ge 5$, and we have already shown this for $k = 4$.

\begin{lemma} \label{mainHua}
Let $i \in \{1,2,3\}$ be as above. Then
\[
\int_\fU |g_i(\alp)^{E+2-\gam} \prod_{2 \le j \le t} g_{E+2j-1}(\alp)^2| \d \alp \ll P^{E+2\Delta-k-\gam}.
\]
\end{lemma}

\begin{proof} Let
\[ f_1(\alp ) = g_i(\alp) \]
and
\[ 
f_j(\alp) = g_{E+2j-1}(\alp) \qquad (2 \le j \le t). 
\]
By Lemma \ref{ClassicalMajorIngredient}, if $\alp \in \fU$ then there exist relatively prime integers $a$ and $q$ such that
\begin{equation} \label{bmain} 0 < q < P^{k\sig(k)}, \qquad |q \alp - a| < P^{k\sig(k)-k} \end{equation}
and
\begin{equation} 
\label{gmain} f_1(\alp) \ll q^{\eps-1/k} P (1+P^k|\beta|)^{-1/k}, 
\end{equation}
where $\bet = \alp -a/q$. For $q \in \bN$ and $a \in \bZ$, denote by $\fU(q,a)$ the set of $\alp \in \fU$ satisfying \eqref{bmain} and \eqref{gmain}. For $q \in \bN$, let $\fU(q)$ denote the union of the sets $\fU(q,a)$ over integers $a$ such that $(a,q) = 1$. Write $\lam_{t+1} = 0$. For $\ell =1,2,\ldots,t$, denote by $\fU^{(\ell)}$ the union of the sets $\fU(q)$ over integers $q$ satisfying \eqref{qrange2}.

Note that
\[
\fU = \bigcup_{\ell \le t} \fU^{(\ell)},
\]
so 
\begin{equation} \label{maindecomp}
\int_\fU |f_1(\alp)^{E+2-\gam} \prod_{2 \le j \le t} f_j(\alp)^2| \d \alp 
\le \sum_{\ell \le t} I_\ell,
\end{equation}
where 
\[ I_\ell = \int_{\fU^{(\ell)}} |f_1(\alp)^{E+2-\gam} \prod_{2 \le j \le t} f_j(\alp)^2| \d \alp
\qquad (1 \le \ell \le t). \]
We now fix $\ell \in \{1,2,\ldots,t\}$. By \eqref{maindecomp}, it remains to show that
\begin{equation} \label{almostdone}
I_\ell \ll P^{E+2\Delta-k-\gam}.
\end{equation}

Let $\alp \in \fU(q,a)$ for some relatively prime integers $q$ and $a$ satisfying $\eqref{qrange2}$, let $2 \le j \le t$, and assume \eqref{hyp}. By Lemma \ref{ClassicalMajorIngredient}, there exist integers $u$ and $r$ satisfying \eqref{rbounds} and \eqref{r_est}. By \eqref{rbounds} and \eqref{bmain} we have
\begin{align}
\notag |a/q - u/r| &\le |\alp-a/q| + |\alp - u/r| \\
\label{fracdif2} &< \frac1{qP^{k-k\sig(k)}} + \frac1{rP^{\lam_j k(1-\sig(k))}}.
\end{align}
As $\lam_j \ge \lam^{t-1}$, we note from \eqref{cond1} that $\lam_j > v/(1-\sig(k))$. We can now use \eqref{qrange2} and \eqref{rbounds}, together with the inequality $v > \sig(k)$, to deduce from \eqref{fracdif2} that $|a/q - u/r| < (qr)^{-1}$. Hence $a/q = u/r$. Thus, recalling that $(a,q)=1$, we see that $q \le r$. Now \eqref{rbounds} and \eqref{r_est} imply \eqref{concs}. 

Let $\alp \in \fU(q)$ for some $q \in \bZ$ satisfying \eqref{qrange2}. From the above discussion, if $2 \le j \le t$ then \eqref{concs} holds whenever \eqref{hyp} holds. Thus, we deduce \eqref{ded1} and \eqref{ded2}. If $q \in \bN$ then there are at most $q+1$ integers $a$ satisfying \eqref{bmain} for some $\alp \in \fU$. By \eqref{qrange2}, applying \eqref{ded1}, \eqref{ded2} and \eqref{gmain} now yields
\[
I_\ell 
\ll P^{E-\gam + 2\Delta + (\eps - 2\sig(k)) (\lam_{\ell+1} + \ldots + \lam_t)}
\sum_{q < P^{\lam_\ell k \sig(k)}} q^{1-(E-\gam+2 \ell)/k+\eps} \bJ,
\]
where
\[
\bJ = \int_0^\infty (1+P^k\bet)^{-(E+2-\gam)/k} \d \bet.
\]
Since $E \ge k$ and $\gamma$ is small, we must have $\bJ \ll P^{-k}$, so
\begin{equation} \label{Il0}
I_\ell \ll P^{E+ 2\Delta - k -\gam  + (\eps - 2\sig(k)) (\lam_{\ell+1} + \ldots + \lam_t)} \sum_{q < P^{\lam_\ell k \sig(k)}} q^{1-(E-\gam+2\ell)/k+\eps}.
\end{equation}

If $E + 2 \ell > 2k$ then the sum is $O(1)$ and so \eqref{almostdone} holds. We may therefore assume in the sequel that 
\begin{equation} \label{lcond}
\ell \le k-E/2.
\end{equation}
From \eqref{Il0} we now have
\[ I_\ell \ll P^{E+ 2\Delta - k -\gam  + \eps - 2\sig(k) \cdot (\lam_{\ell +1} + \ldots + \lam_t)
+ \lam_\ell k\sig(k) \cdot (2-(E+2 \ell-\gam)/k)}. \]
As $\gam$ and $\eps$ are small, it therefore suffices to show that
\begin{equation} \label{laststop}
\lam_\ell  k (2-(E+2\ell )/k) < 2(\lam_{\ell +1}+\ldots + \lam_t).
\end{equation}

As an intermediate step, we show that
\begin{equation} \label{inter}
1- \frac{1-v}{2\lam} > \lam^k.
\end{equation}
For $4 \le k \le 8$, we simply check this directly. Now suppose $k \ge 9$. As $\lam > (k-1)/k$, we have
\begin{align*}
1- \frac{1-v}{2\lam} &> 1-\frac{(1-v)k}{2(k-1)} > 1-\frac{k}{2(k-1)} \ge 1-9/16 > e^{v-1} \\
& \ge (1-(1-v)/k)^k = \lam^k.
\end{align*}
Thus we have \eqref{inter} for all integers $k \ge 4$. 

By \eqref{cond2} and \eqref{lcond} we have 
\[ t-\ell \ge t - (k-E/2) = (2t+E)/2 - k \ge k. \]
This and \eqref{inter} give
\[ \lam^{t-\ell} \le \lam^k < 1- \frac{1-v}{2\lam}. \]
Now
\[ \frac k2 <  \frac {k\lam}{1-v} (1-\lam^{t-\ell}) = \frac \lam {1-\lam} (1- \lam^{t- \ell}), \]
as we recall \eqref{henceforth}. Hence
\[
\lam_\ell k /2  < \lam_\ell (\lam + \lam^2 \ldots + \lam^{t- \ell}) = \lam_{\ell +1} + \ldots + \lam_t.
\]
Since $E \ge k$, this gives \eqref{laststop}, completing the proof of Lemma \ref{mainHua}.
\end{proof}

The remainder of the proof is identical to that of \eqref{basicBounds} for $\iota=1$, which is given in \S \ref{s1bitBasic}. Thus we have established Theorem \ref{mainThm} which, as discussed, produces the data in Table \ref{table1}.

\bibliographystyle{amsrefs}
\providecommand{\bysame}{\leavevmode\hbox to3em{\hrulefill}\thinspace}

\end{document}